\documentclass[12pt,twoside]{article}
 
\usepackage[margin=1in]{geometry} 
\usepackage{amsmath,amsthm,amssymb}
\usepackage{amsfonts}
\usepackage{hyperref}
\usepackage{fancyhdr}
\usepackage{graphicx}
\usepackage{tikz}
\usepackage{float}
\usepackage{indentfirst}
\usetikzlibrary{decorations.pathreplacing}
\newenvironment{theorem}[2][Theorem]{\begin{trivlist}
\item[\hskip \labelsep {\bfseries #1}\hskip \labelsep {\bfseries #2.}]}{\end{trivlist}}

\newenvironment{definition}[2][Definition]{\begin{trivlist}
\item[\hskip \labelsep {\bfseries #1}\hskip \labelsep {\bfseries #2.}]}{\end{trivlist}}
\newtheorem{remark}{Remark}[section]
\begin{document}
 
 
\title{The Top Dimensional Singular Set $\text{sing}_{*}u$} 
\author{
  Owen Drummond \\
  Department of Mathematics, Rutgers University \\
  \texttt{owen.drummond@rutgers.edu} \\
  \\
  \small{Advised by: Natasa Sesum} \\
}
\maketitle
\begin{abstract}
    In this survey, we examine the properties of the top dimensional singular set $\text{sing}_{*}(u)$, including several theorems, geometric properties, and relation to homogeneous degree zero minimizers.
\end{abstract}
\tableofcontents
\section{Introduction}
In geometric analysis, energy minimizing maps have been a prominent area of study for nearly 40 years, particularly in the 1980's and 1990's. One specific area of interest within this topic is the structure of the singular set $\text{sing}(u)$, where $u \in W^{1,2}(\Omega; N)$ is energy minimizing, $\Omega \subset \mathbb{R}^n$, and N is a Riemannian target manifold isometrically embedded into some Euclidean Space $\mathbb{R}^p$. For understanding the behavior of energy minimizing maps near singular points, one investigates the nature of tangent maps, that is, maps that the extrapolate the local behavior of an energy minimizing map to a global scale by convergence along a subsequence. Such convergence is guaranteed by several key compactness results about the Sobolev Space $W^{1,2}$, yet much remains unknown about tangent maps. The question of uniqueness of tangent maps is classically quite a difficult problem, both in this context, and in the geometric flows as well, namely mean curvature flow and Ricci flow.

In his notable 1996 work \textit{Theorems on Regularity and Singularity of Energy Minimizing Maps}\cite{simon1996theorems}, Leon Simon investigates this question of uniqueness from the perspective of geometric analysis and measure theory. In Chapter 3 of his work, which is concerned with the structure of $\text{sing}(u)$, Simon introduces the Top Dimensional Singular Set $\text{sing}_{*}(u)$, which is concerned with tangent maps of dimension $(n-3)$. This set is fundamental in his discussion of the uniqueness of tangent maps under certain prescribed conditions, such as in Section 3.8. This paper follows Simon's book and will serve as an exposition of this set $\text{sing}_{*}(u)$, including key properties and theorems. 
\section{Preliminaries}
One can reference my paper \textit{Geometric Analysis of Energy Minimizing Maps}\cite{drummond2024geo} to study these preliminaries in greater detail. We will assume that the reader is familiar with basic Real Analysis and PDE Theory. The general framework which we will operate under remains the same: let $\Omega$ is an open subset of $\mathbb{R}^n$, for $n \geq 2$, and that $N$ is a compact, $C^{\infty}$  Riemannian Manifold which is isometrically embedded into $\mathbb{R}^p$ for some $p$. Assume $u: \Omega \rightarrow \mathbb{R}^p$ such that $u(\Omega) \subset N$. For $\Omega$ and $N$ defined in this way, we denote the Sobolev Space $W_{\text{loc}}^{1,2}(\Omega; N)$ to be the set of functions $u \in W_{\text{loc}}^{1,2}(\Omega; \mathbb{R}^p)$ with $u(x) \in N$ a.e. $x \in \Omega$.
\subsection{Energy and \texorpdfstring{$\epsilon$}{text}-regularity}
\begin{definition}{2.1.1 (The Energy Functional)}
The energy $\mathcal{E}_{B_\rho (y)}(u)$ for a function $u \in W_{\text{loc}}^{1,2}(\Omega; N)$ in a ball $B_\rho (y) = \{ x: |x-y|<\rho \}$ with $\bar{B}_{\rho} \subset \Omega$ is 
\[
\mathcal{E}_{B_\rho (y)}(u)=\int_{B_\rho (y)} |Du|^2
\]
\end{definition}
\begin{definition}{2.1.2 (Energy Minimizing Maps)}
A map $u \in W^{1,2}(B_{\rho}(y); N)$ is called \textit{energy minimizing} if for each ball $B_\rho (y) \subset \Omega$
\[
\mathcal{E}_{B_\rho (y)}(u) \leq \mathcal{E}_{B_\rho (y)}(w) 
\]
for every $w \in W^{1,2}(B_\rho (y); N)$ with $w=u$ in a neighborhood of $\partial B_{\rho}(y)$. 
\end{definition}
This gives way to the following fact: a map $\psi$ is harmonic if it is a critical point of the energy functional. More specifically, let $\bar{B}_\rho(y) \subset \Omega$ as before, and for any $\delta > 0$, we have the 1-parameter family $\{\psi_s \}_{s \in (-\delta, \delta)}$ such that $\psi_0=\psi$, $D\psi_s \in L^2(\Omega)$, and $\psi_s \equiv \psi$ in a neighborhood of $\delta B_{\rho}(y)$ for all $s \in (-\delta, \delta)$. Thus, any such $\psi_s$ is harmonic if and only if
\[
\frac{d\mathcal{E}_{B_\rho(y)}(\psi_s)}{ds}|_{s=0}=0
\]
this same equation leads to the Euler-Lagrange equations for the map, which are harmonic map equations. As a consequence, harmonic maps are also local minimizers of the energy functional, that is, in an $\epsilon$-neighborhood of a map, no other map has lower energy than a harmonic map. However, there is no such global condition for harmonic maps as energy minimizers.
\par Here is an important theorem of Schoen and Uhlenbeck \cite{schoen1982regularity} that bears on the regularity of energy minimizing maps and will be employed several times throughout this paper:
\begin{theorem}{2.1.3 ($\epsilon$-Regularity Theorem)}
    Let $\Lambda > 0$, $\theta \in (0, 1)$. There exists $\varepsilon = \varepsilon(n, \Lambda, N, \theta) > 0$ such that if $u \in W^{1,2}(\Omega; N)$ is energy minimizing on $B_R(x_0) \subset \Omega$ and if
    \[
    R^{2-n} \int_{B_R(x_0)} |Du|^2 \leq \Lambda \quad \text{and} \quad R^{2-n} \int_{B_R(x_0)} |u - \lambda_{x_0, R}|^2 < \varepsilon^2,
    \]
    then there holds $u \in C^\infty(B_{\theta R}(x_0))$, and for $j = 1, 2, \ldots$ we have the estimates
    \[
    R^j \sup_{B_{\theta R}(x_0)} |D^j u| \leq C (R^{-n} \int_{B_R(x_0)} |u - \lambda_{x_0, R}|^2)^{1/2},
    \]
    where $C$ depends only on $j$, $\Lambda$, $N$, $\theta$, and $n$.
\end{theorem}
Essentially, this theorem states that if the average deviation from an approximate mean of an energy miniziming map is sufficiently small within some ball $B_{R}(x_0)$, then $u$ must be quite regular within a smaller ball $B_{\theta R}(x_0)$. This is foundational as it bridges the gap between local behavior and global regularity, directly impacting our understanding of singularities in geometric flows.

\subsection{Tangent Maps and sing(u)}

Next, we introduce the notion of tangent maps. In essence, these maps extrapolate local behavior of some energy minimizing map to a global function. Given $u:\Omega \rightarrow \mathbb{R}^p$ and $B_{\rho_0}(y)$ such that $\bar{B}_{\rho_0}(y) \subset \Omega$, and for any $\rho>0$, consider the scaling function $u_{y,\rho}$ given by 
\[
u_{y,\rho}(x)=u(y+\rho x)
\]
Note that on $B_{\rho_0}(0)$, $u_{y,\rho}$ is well defined. For $\sigma > 0$ and $\rho < \frac{\rho_0}{\sigma}$, after making a change of variables with $\Tilde{x}=y+\rho x$ in the energy integral for $u_{y,\rho}$ and noting that $Du_{y,\rho}(x)=\rho(Du)(y+\rho x)$, we have
\begin{align}
\sigma^{2-n} \int_{B_\sigma(0)} |Du_{y,\rho}|^2 = (\sigma \rho)^{2-n} \int_{B_{\sigma \rho}(y)} |Du|^2 \leq \rho_0^{2-n} \int_{B_{\rho_0}(y)} |Du|^2
\end{align}
by the Monotonicity Formula. Therefore, if $\rho_j \downarrow 0$, then $\limsup_{j \rightarrow \infty} \int_{B_{\sigma}(0)} |Du_{y,\rho_j}|^2 < \infty$ for all $\sigma > 0$, and so by the compactness, there is a subsequence $\rho_{j'}$ such that $u_{y,\rho_{j'}} \rightarrow \varphi$ locally in $\mathbb{R}^n$ w.r.t. the $W^{1,2}$-norm.

\begin{definition}{2.2.1 (Tangent Map)}
    Any $\varphi$ obtained this way is called a \textit{tangent map of u at y}. Moreover, $\varphi: \mathbb{R}^n \rightarrow N$ is an energy minimizing map with $\Omega = \mathbb{R}^n$.
\end{definition}
Simon's approach to understanding the properties of tangent maps is studying \textit{degree zero minimizers}, a broader class of functions which tangent maps belong.

\begin{definition}{2.2.2 (Degree Zero Minimizer)}
    A function $\varphi: \mathbb{R}^n \rightarrow N$ is a \textit{degree zero minimizer} if $\varphi(\lambda x) \equiv \varphi(x)$ for all $\lambda > 0$, $x \in \mathbb{R}^n$
\end{definition}

More on the properties of degree zero minimizers can be found in my previous paper. 
\begin{definition}{2.2.3 (Regular and Singular Sets)}
    If $u \in W^{1,2}(\Omega: \mathbb{R}^p)$, then 
\[
\text{reg}(u) := \{ x \in \Omega : u \text{ is } C^{\infty} \text{ in a neighborhood of } x \}
\]
is the regular set of $u$, and 
\[
\text{sing}(u):=\Omega \setminus \text{reg}(u)
\]
is the singular set of $u$.
\end{definition}
The singular set $\text{sing}(u)$ is topic of great interest in the study of energy minimizing maps due to its obscure geometric structure. Certain techniques from geometric measure theory, primarily rectifiabilty and gap measures, have proven to be particularly effective, and are explored in Simon's work.

\section{Definition of \texorpdfstring{$\text{sing}_{*}u$}{text}}
As seen in my previous work (p. 13)\cite{drummond2024geo}, $\text{sing}(u)$ consists of only isolated points, that is, $\mathcal{H}^{n-2}(\text{sing}(u))=0$. Thus, $\text{sing}(u)$ is best understood in the context of codimension 3, that is, dimension $n-3$. In 2015, Aaron Naber and Daniele Valtorta\cite{naber2015rectifiable} have proven the remarkable result that for a smooth, compact target manifold $N$, $\mathcal{H}^{n-3}(\text{sing}(u)) < \infty$ uniformly, and this result was previously unknown for some time. This was proven using modern technique of quantifiable rectifiability in geometric measure theory. 
\par Here, we are concerned with target manifolds $N$ of dimension $n-3$. However, this discussion is still meaningful in the case of dimension $m \leq n-4$ as long as $N$ happens to be such that all homogeneous degree zero (tangent) maps $\varphi \in W_{\text{loc}}^{1,2}(\mathbb{R}^n;N)$ of $u$ satisfy $\dim S(\varphi) \leq m$.
\begin{definition}{3.1 (Top Dimensional Singular Set)}
    The top dimensional part $\text{sing}_{*}(u)$ of $\text{sing}(u)$ is the set of points $y \in \text{sing}(u)$ such that some tangent map $\varphi$ of $u$ at $y$ has $\dim S(\varphi) = n-3$.
\end{definition}
from this we obtain the immediate result:
\begin{theorem}{3.2}
    For $\Omega \subset \mathbb{R}^n$ open and $u \in W^{1,2}(\Omega; N)$,
    \[
    \dim(\text{sing}(u) \setminus \text{sing}_{*}(u)) \leq n-4
    \]
\end{theorem}
\begin{proof}
    Given the fact that $\text{sing}(u)=\mathcal{S}_{n-3}$, for $y \in \text{sing}(u) \setminus \text{sing}_{*}(u)$, any tangent map $\varphi$ of $u$ at $y$ must satisfy 
    \[
    \dim S(\varphi) = \dim \{y \in \mathbb{R}^n: \Theta_{\varphi}(y)=\Theta_{\varphi}(0) \} = n-3
    \]
    Putting these two facts together, it must be the case that this same tangent map $\varphi$ associated with $y \in \text{sing}(u) \setminus \text{sing}_{*}(u)$ must have $\dim S(\varphi) \leq n-4$, and hence 
    \[
    \text{sing}(u) \setminus \text{sing}_{*}(u) \subset \mathcal{S}_{n-4}
    \]
    Now by Lemma 1 of Section 3.4 of Simon (p. 54)\cite{simon1996theorems}, which states that for each $j=0,...,n-3$, $\dim \mathcal{S}_j \leq j$, he result follows, namely
    \[
    \dim(\text{sing}(u) \setminus \text{sing}_{*}(u)) \leq \dim(\mathcal{S}_{n-4}) \leq n-4
    \]
    as desired.

\end{proof}
To continue our discussion of $\text{sing}_{*}(u)$, we will examine homogeneous degree zero minimizers $\varphi: \mathbb{R}^n \rightarrow N$ with $\dim S(\varphi)=n-3$.

\section{H.D.Z.M.'s with \texorpdfstring{$\dim S(\varphi)=n-3$}{text}}
Let $\varphi: \mathbb{R}^n \rightarrow N$ be an arbitrary homogeneous degree zero minimizer (H.D.Z.M.) with $\dim S(\varphi)=n-3$. Now modulo an orthogonal transformation of $\mathbb{R}^n$ which takes $S(\varphi)$ to $\{0 \} \times \mathbb{R}^{n-3}$, and so we can write $\varphi(x,y) \equiv \varphi_{0}(x)$, where $(x,y) \in \mathbb{R}^n$ with $x \in \mathbb{R}^3$ and $y \in \mathbb{R}^{n-3}$, and where $\varphi_0: \mathbb{R}^3 \rightarrow N$ is a H.D.Z.M. 
\begin{theorem}{4.1}
    $\varphi_0|_{S^2}: S^2 \rightarrow N$ is a $C^\infty$ harmonic map.
\end{theorem}
\begin{proof}
    We claim that $\text{sing}\varphi_0=\{0 \}$, which is sufficient to establish the smoothness of $\varphi_0$. First, note that $\text{sing}\varphi_0 \supset \{0\}$, otherwise $\varphi_0$, and by extension, $\varphi$ would be constant, yet $\dim S(\varphi) = n-3$, which is a contradiction. Now suppose that $\xi \in \text{sing}\varphi_0$ and $\xi \neq 0$. By the homogeneity of $\varphi_0$, we have that
    \[
    \{\lambda \xi: \lambda>0 \} \subset \text{sing}\varphi_0 \implies \{(\lambda \xi, y): \lambda>0, y \in \mathbb{R}^{n-3} \} \subset \text{sing}\varphi
    \]
    Note that the set $\{(\lambda \xi, y): \lambda>0, y \in \mathbb{R}^{n-3} \}$ is a half-plane of dimension (n-2), and hence 
    \[
    \mathcal{H}^{n-2}(\text{sing}\varphi) \geq \mathcal{H}^{n-2}(\{(\lambda \xi, y): \lambda>0, y \in \mathbb{R}^{n-3} \})=\infty
    \]
    and thus $\mathcal{H}^{n-2}(\text{sing}\varphi)=\infty$. However, since $\text{sing}(u)$ consists of only isolated points, for any energy minimizing map $u \in W^{1,2}(\Omega; N)$, where $\Omega \subset \mathbb{R}^n$, $\mathcal{H}^{n-2}(\text{sing}(u))=0$. Thus, $\mathcal{H}^{n-2}(\text{sing}\varphi)=\infty$ is a contradiction, so $\xi=0$, which establishes that $\text{sing}\varphi_0 = \{0\}$. Therefore, $\varphi_0|_{S^2}$ is away from the singularity at $0$, so the regularity set is the entire domain $S^2$, so $\varphi_0|_{S^2}: S^2 \rightarrow N$ is smooth and harmonic.
\end{proof}
Next is an important theorem related the boundedness of the $L^2$-norm of the gradient of $\varphi$, with the general finiteness of the gradient of $\varphi_0$ constructed in this way.
\begin{theorem}{4.2}
    If $\varphi^{(j)}$ is any sequence of H.D.Z.M.'s with $\varphi^{(j)}(x,y) \equiv \varphi_0^{(j)}(x)$ for each $j \in [1,\infty)$, and if $\displaystyle\limsup_{j \to \infty}\int_{B_1(0)} |D\varphi^{(j)}|^2 < \infty$, then 
    \[
    \displaystyle\limsup_{j \to \infty}[\sup_{S^2} |D^{\ell}\varphi_0^{(j)}|] < \infty \hspace{0.75cm} \forall \ell \geq 0
    \]
\end{theorem}
\begin{proof}
    By compactness, there exists a subsequence of $\varphi^{(j)}$, say $\varphi^{(j')}$, such that $\varphi^{(j')}\rightarrow \varphi$, where $\varphi$ is a H.D.Z.M. with $\varphi(x,y)=\varphi_0(x)$, and by the previous theorem, $\varphi_0|_{S^2} \in C^\infty$. Thus, for $z \in \mathbb{R}^n \setminus (\{0 \} \times \mathbb{R}^{n-3})$, that is, a point in $z=(z_1,z_2, ..., z_n) \in \mathbb{R}^n$ such that $(z_1, z_2, z_3) \neq (0,0,0)$, and a positive $\sigma < \text{dist}(z,\{0 \} \times \mathbb{R}^{n-3})$, we have 
    \[
    |\varphi(x)-\varphi(z)| \leq C\sigma < C \text{dist}(z,\{0 \} \times \mathbb{R}^{n-3}) = C\sqrt{z_1^2+z_2^2+z_3^2}
    \]
    by the smoothness of $\varphi$. This formulation admits a kind of Lipschitz continuity. Hence,
    \[
    \int_{B_\sigma(z)}|\varphi^{(j')}-\varphi(z)|^2 \leq \int_{B_\sigma(z)}|\varphi^{(j')}-\varphi|^2+C\sigma^2
    \]
    Now by the $\epsilon$-regularity theorem, and by letting $\sigma$ be sufficiently small, and for $j'$ sufficiently large, the convergence of $\varphi^{(j')}$ to $\varphi$ is with respect to the $C^k$ norm for each $k$ on compact subsets $K \subset \mathbb{R}^n \setminus (\{0 \} \times \mathbb{R}^{n-3})$. In light of the fact that the subsequence $\varphi^{(j')}$ was arbitrary, it follows that for all H.D.Z.M.'s $\varphi(x,y) \equiv \varphi_0(x)$, where $x \in \mathbb{R}^3, y \in \mathbb{R}^{n-3}$, with $\int_{S^2} |D\varphi_0|^2 \leq \Lambda$, we have that 
    \[
    \sup_{S^2}|D^{\ell} \varphi_0|^2 \leq C, \hspace{0.75cm} \ell=1,2,...
    \]
    and $C$ depends on the target manifold $N$, the constant $\Lambda$, and the index $\ell$.
\end{proof}
This theorem will be a powerful resource for investigating the geometric picture of $\text{sing}_{*}(u)$.
\section{The Geometry of \texorpdfstring{$\text{sing}_{*}(u)$}{text}}
For the general framework, let $K \subset \Omega$, and $z \in \text{sing}_{*}u \cap K$, and let $\varphi$ be a tangent map of $u$ at $z$ with $\dim S(\varphi) = n-3$. Also, let $\varphi(x,y)=\varphi_0(x)$, where $x \in \mathbb{R}^3, y \in \mathbb{R}^{n-3}$ be as before. Leveraging the properties of $\text{sing}_{*}(u)$, which includes the existence of a tangent map $\varphi$ of $u$ at every point $z \in \text{sing}_{*}(u)$ that can approximate $u$ in the $L^2$ over smaller and smaller balls around $z$, we have that for $\varphi^{(z)}(x,y) \equiv \varphi((x,y)-z)$, there is a sequence $\rho_j \downarrow 0$ such that
\begin{align}
\displaystyle\lim_{j \to \infty} \rho_j^{-n} \int_{B_{\rho_j}(z)} |u-\varphi^{(z)}|^2 = 0
\end{align}
As a result, for $\rho=\rho_j$ with $j$ sufficiently large, we can control the scaled $L^2$-norm given by
\[
\rho^{-n}\int_{B_\rho(z)} |u-\varphi^{(z)}|^2
\]
by making it as small as we please. From this we obtain the following theorem which describes the geometric picture of $\text{sing}_{*}(u)$. 
\begin{theorem}{5.1}
For any H.D.Z.M.'s $\varphi: \mathbb{R}^n \rightarrow N$, and any ball $B_{\rho_0}(z)$ with $\bar{B}_{\rho_0}(z) \subset \Omega$ we have the estimate
\[
\textrm{sing}(u) \cap B_{\rho/2}(z) \subset {x: \textrm{dist}(x, (z+\{0 \} \times \mathbb{R}^{n-3})) < \delta(\rho)\rho} \hspace{0.5cm} \forall \rho \leq \rho_0
\]
where
\[
\delta(\rho)=C \left( \rho^{-n} \int_{B_\rho(z)} |u-\varphi^{(z)}|^2 \right)^{\frac{1}{n}}
\]
and $C$ only depends on $n,N,\Lambda$, and $\Lambda$ is any upper bound for $\rho_0^{2-n}\int_{B_{\rho_0}(z)}|Du|^2$.
\end{theorem}

\begin{figure}[H] 
    \centering
    \includegraphics[width=0.5\linewidth]{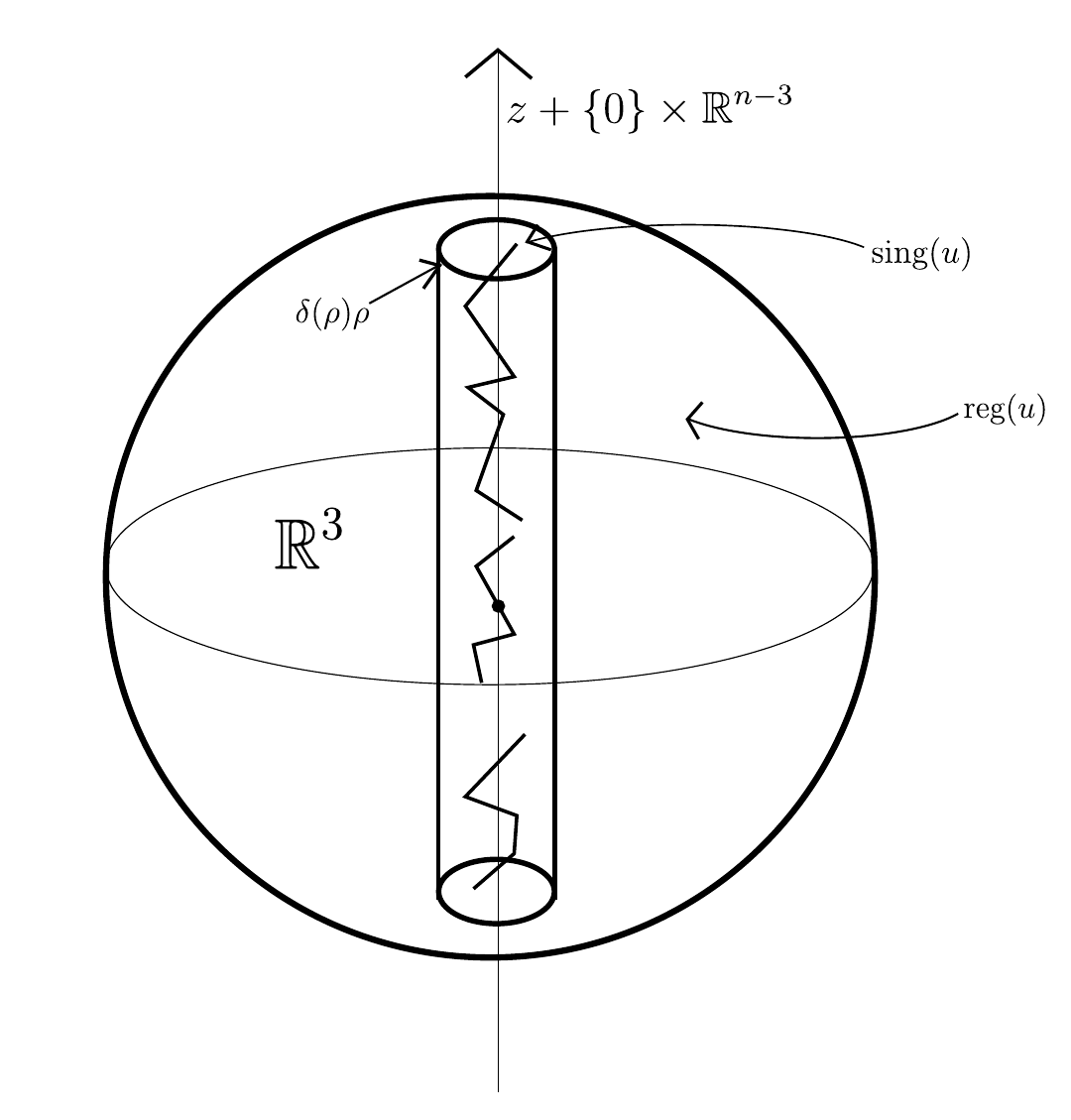}
    \caption{\textit{The Geometric Picture of} $\text{sing}_{*}u$}
    \label{fig:my_label}
\end{figure}

\begin{remark}
In light of the fact that we can control the scaled $L^2$-norm as seen in (2), one might suggest that $\text{sing}_{*}(u)$ in contained in a Lipschitz manifold, if not a $C^1$ manifold of dimension $n-3$. However, given the definition of $S(\varphi)$, $\delta(\rho)$ is small only when $\rho$ is proportionally close to one of the $\rho_j$. Without further information, we cannot say much about the geometric structure of $\text{sing}_{*}(u)$.
\end{remark}

\begin{proof}
Without loss of generality, assume $z=0$. Let $\rho<\rho_0$, and define $w=(\xi,\eta) \in \text{sing}_{*}(u) \cap B_{\rho/2}(0)$. Choose $\sigma=\beta_0|\xi|$, and $\beta_0 \leq \frac{1}{2}$. By the $\epsilon$-regularity theorem, $\exists \epsilon_0=\epsilon_0(n,N,\Lambda)>0$ such that
\begin{align}
\epsilon_0 \leq \sigma^{-n}\int_{B_\sigma(w))}|u-\varphi(w)|^2 \leq 2\sigma^{-n}\int_{B_\sigma(w)}|u-\varphi|^2+2\sigma^{-n}\int_{B_\sigma(w)}|\varphi-\varphi(w)|^2
\end{align}
which follows from triangle inequality. Now by Theorem 4.2, we have that
\[
|D\varphi_0(x)|\leq C|x|^{-1}
\]
and $C$ depends only on $N,\Lambda$, hence
\[
|\varphi(w)-\varphi(x)| \leq C|\xi|^{-1}\sigma \leq C\beta_0
\]
for $x \in B_{\sigma}(w)$. Now by (3):
\[
\epsilon_0 \leq 2\beta_0^{-n}|\xi|^{-n}\int_{B_\rho(0)}|u-\varphi|^2+C\beta_0^2
\]
So by selecting $C\beta_0^2 \leq \frac{1}{2}\epsilon_0$ and multiplying through by $|\xi|^n$ on both sides, we have 
\begin{align*}
    |\xi|^n \leq C(\rho^{-n}\int_{B_\rho(0)}|u-\varphi|^2)\rho^n \\
    \implies |\xi| \leq C(\rho^{-n}\int_{B_\rho(0)}|u-\varphi|^2)^{\frac{1}{n}}
\end{align*}
where $C$ is some constant depending only on $n,N$, and $\Lambda$, which is the inequality we desired to show.
\end{proof}

\end{document}